\documentclass[12pt]{amsart}
\usepackage{amsfonts}
\usepackage{amsthm}
\usepackage{amsmath}
\usepackage{epsf}
\usepackage[dvips]{epsfig}
\usepackage{latexsym}
\usepackage{psfrag}

\usepackage[usenames]{color} 
\usepackage{times}
\usepackage{amsfonts}
\usepackage{amsthm}

\usepackage{times}
\usepackage{latexsym,color}
\usepackage{amssymb}
\usepackage{graphicx}
\usepackage{epsfig}
\usepackage{gastex}
\advance \topmargin by -\headheight
\advance \topmargin by -\headsep
\evensidemargin \oddsidemargin
\makeatletter
\def\Ddots{\mathinner{\mkern1mu\raise\p@
\vbox{\kern7\p@\hbox{.}}\mkern2mu
\raise4\p@\hbox{.}\mkern2mu\raise7\p@\hbox{.}\mkern1mu}}
\makeatother

\def \ind{_{n \in {\mbox{\rm {\scriptsize I$\!$N}}}}}

\newtheorem{theorem}{Theorem}
\newtheorem{lm}[theorem]{Lemma}
\newtheorem{cl}[theorem]{Corollary}
\newtheorem{pp}[theorem]{Proposition}

\newtheorem{ex}{Example}
\newtheorem{df}{Definition}

\begin{document}

\title{Clustering words}

\author[S. Ferenczi]{S\'ebastien Ferenczi}
\address{Institut de Math\'ematiques de Luminy\\ CNRS - FRE 3529\\
Case 907, 163 av. de Luminy\\
F13288 Marseille Cedex 9, France  \indent and \indent
 F\'ed\'eration de Recherche des Unit\'es de Math\'ematiques de Marseille\\ 
 CNRS - FR 2291}
\email{ferenczi@iml.univ-mrs.fr}

\author[L.Q. Zamboni]{Luca Q. Zamboni}
\address{Institut Camille Jordan\\
Universit\'e Claude Bernard Lyon 1\\
43 boulevard du 11 novembre 1918\\
F69622 Villeurbanne Cedex, France \indent and \indent Department of Mathematics and
                        Turku Centre for Computer Science,
                       University of Turku,
                        20014 Turku, Finland.}
\email{zamboni@math.univ-lyon1.fr}

\date{April 6, 2012}

\begin{abstract} We characterize words  which cluster under the Burrows-Wheeler transform as those words  $w$  such that $ww$ occurs in a trajectory of an interval exchange transformation, and build examples of clustering words. \end{abstract}

\maketitle

In 1994 Michael Burrows and David Wheeler \cite{bw} introduced a  transformation on words which proved very powerful in
data compression.
The aim of the present note is to characterize those words which cluster under the Burrows-Wheeler transform, that is to say which are transformed into such expressions as $4^a3^b2^c1^d$ or $2^a5^b3^c1^d4^e.$ Clustering words on a binary alphabet have already  been extensively studied (see for instance in \cite{jz,rr3}) and identified as particular factors of the Sturmian words. Some generalizations to $r$ letters appear in \cite{rr3}, but it  had not yet been observed that clustering words are intrinsically related to interval exchange transformations (see Definitions \ref{d1} and  \ref{d2} below). This link comes essentially from the fact  that the array of conjugates used to define the Burrows-Wheeler transform gives rise to a {\it discrete} interval exchange transformation sending its first column to its last column. It turns out that the converse is also true: interval exchange transformations generate clustering words. Indeed  we prove that clustering words are exactly those words $w$ such that $ww$ occurs in a trajectory of an interval exchange transformation. On a binary letter alphabet, this condition amounts to saying that  $ww$ is a factor of an infinite Sturmian word. We end the paper by some examples and questions on how to generate clustering words.\\

This paper began during a workshop on board Via Rail Canada train number 2. We are grateful to  Laboratoire International Franco-Qu\' eb\' ecois de Recherche en Combinatoire (LIRCO) for funding and Via for providing optimal working conditions. The second author is partially supported by a grant from the Academy of Finland.

\section{Definitions}

Let  $A=\{a_1<a_2< \cdots <a_r\}$ be an ordered alphabet and $w=w_1\cdots w_n$  a {\em primitive} word on the alphabet $A,$ i.e. $w$ is not a power of another word. For simplification we suppose that {\em each letter of $A$ occurs in $w$}. \\

The {\em Parikh vector} of $w$ is the integer vector $(n_1,\ldots ,n_k)$ where $n_i$ is the number of occurrences of $a_i$ in $w$. The {\em (cyclic) conjugates} of $w$ are  the words $w_i\cdots w_nw_1\cdots w_{i-1}$, $1\leq i\leq n.$ As $w$ is primitive,  $w$ has precisely $n$-cyclic conjugates.  Let $w_{i,1}w_{i,2}\cdots w_{i,n}$ denote the $i$-th  conjugate of $w$ where the $n$-conjugates of $w$ are ordered in  ascending lexicographical order.
 Then the {\em  Burrows-Wheeler transform} of $w$, denoted by $B(w)$, is the word $w_{1,n}w_{2,n}\cdots w_{n,n}.$
 In other words, $B(w)$ is obtained from $w$ by first ordering its cyclic conjugates in ascending order in a rectangular array, and then reading off the last column.
 We say $w$ is $\pi${\em -clustering} if $B(w)=a_{\pi 1}^{n_{\pi 1}}\cdots a_{\pi r}^{n_{\pi r}}$, where $\pi\neq Id$ is a permutation on $\{1, \ldots ,r\}$. We say $w$ is {\em perfectly} clustering if it is $\pi$-clustering for $\pi i=r+1-i$, $1\leq i\leq r$.

\begin{df}\label{d1}\rm{A (continuous) {\it $r$-interval exchange transformation}
 $T$
with probability vector \\
$(\alpha _1,\alpha _2,\ldots ,\alpha _r)$,
 and
permutation $\pi$ is defined on the interval $[0,1[$, partitioned into $r$ intervals
$$\Delta_i=\left[ \sum_{j<i}\alpha_{j}
,\sum_{j\leq i}\alpha_{j}\right[,$$
by $$Tx=x+\tau_i \quad \mbox{when} \quad x\in\Delta_i,$$ where $\tau_i
=\sum_{\pi^{-1}(j)<\pi^{-1}(i)}\alpha_{j}-\sum_{j<i}\alpha_{j}$. }\end{df}

Intuitively this means that the intervals $\Delta_i$ are re-ordered by $T$ following the permutation $\pi$. We refer the reader to  \cite{v} which constitutes a classical course on general interval exchange transformations and contains many of the technical terms found in Section \ref{s3}  below. Note that our use of the word ``continuous" does not imply that $T$ is a continuous map on $[0,1[$ (though it can be modified to be made so); it is there to emphasize the difference with its discrete analogous.

\begin{df}\label{d2} \rm{A {\it discrete $r$-interval exchange transformation}
 $T$
with length vector $(n _1,n _2,\ldots ,n_r)$,
 and
permutation $\pi$ is defined on  a set of $n_1+\cdots +n_r$ points $x_1$, \ldots , $x_{n_1+\cdots +n_r}$
partitioned into $r$ intervals
$$\Delta_i=\{x_k,  \sum_{j<i}n_{j} < k \leq
\sum_{j\leq i}n_{j}\}$$
 by
$$Tx_k=x_{k+s_i} \quad \mbox{when} \quad x_k\in\Delta_i,$$ where $s_i=\sum_{\pi^{-1}(j)<\pi^{-1}(i)}n_{j}-\sum_{j<i}n_{j}$.}\end{df}

We recall the following notions, defined for any transformation $T$ on a set $X$ equipped with a partition $\Delta_i$, $1\leq i\leq r$.

\begin{df}  \rm{The {\it trajectory}  of a point $x$ under $T$ is the infinite
sequence
$(x_{n})\ind $ defined by $x_{n}=i$ if $T^nx$ belongs to
$\Delta_i$, $1\leq i\leq r$.  The mapping $T$  is {\it minimal} if whenever $E$ is a nonempty closed subset of $X$ and $T^{-1}E=E$, then $E=X$.}\end{df}

\section{Main result}
\begin{theorem}\label{t1} Let $w=w_1\cdots w_n$ be a primitive word on $A=\{ 1,\ldots ,r\}$, such that every letter of $A$ occurs in $w$. The following are equivalent:
\begin{enumerate}  \item  $w$ is $\pi$-clustering, 
\item $ww$ occurs in a trajectory of a minimal discrete $r$-interval exchange transformation with permutation $\pi$, 
\item $ww$ occurs in a trajectory of a discrete $r$-interval exchange transformation with permutation $\pi$, 
\item $ww$ occurs in a trajectory of a continuous $r$-interval exchange transformation with permutation $\pi$.
\end{enumerate}
\end{theorem}

\begin{proof}{($(2)$, $(3)$ or  $(4)$  implies $(1)$)}
By assumption there exists a point $x$ whose initial trajectory of length $2n$ is the word $ww.$ Consider the set $E=\{Tx, T^2x, \ldots , T^nx\}.$ Then for each $y\in E,$ the initial trajectory of $y$ of length $n,$ denoted $O(y),$
 is a cyclic conjugate of $w$. 
 
 Suppose $y$ and $z$ are in $E$, and $y$ is  to the left of $z$ (meaning $y<z.$). Let $j$ be the smallest nonnegative integer such that $T^jy$ and $T^jz$ are not in the same $\Delta_i.$ Then $T^{j}y$ is to the left of  $T^{j}z$, either because $j=0$ or because $T$ is increasing on each $\Delta_i.$ Thus $O(y)$ is lexicographically smaller that $O(z).$  
 
Thus $B(w)$ is obtained from the last letter  $l(y)$ of $O(y)$ where the points $y$ are ordered from left to right. But $l(y)$  is the label of the interval $\Delta_i$ where $T^{n-1}y$, or equivalently $T^{-1}y$, falls. Thus by definition of $T$, if $y$ is to the left of $z$ then $\pi^{-1}(l(y))\leq \pi^{-1}(l(z))$, and if $y'$ is between $y$ and $z$ with $l(y)=l(z)$, then $l(y')=l(y)=l(z)$, hence the claimed result.\end{proof}
 
\begin{proof}{($(2)$ implies  $(3)$ implies $(4)$)}
The first implication is trivial. The  second follows from the fact that the trajectories of the discrete $r$-interval exchange transformation with length vector $(n _1,n _2,\ldots ,n_r)$,
 and
permutation $\pi$, and of the continous $r$-interval exchange transformation with probability vector $(\frac{n_1}{n_1+\cdots +n_r}, \ldots ,\frac{n_r}{n_1+\cdots + n_r})$ 
 and
permutation $\pi$ are the same. We note that this continuous interval exchange transformation is never minimal, while the discrete one may be. \end{proof}

We now  turn to the proof of the converse, which uses a succession of lemmas. Throughout this proof, unless otherwise stated, a given word $w$ is a primitive word on $\{ 1,\ldots ,r\}$, and  every letter of $\{ 1,\ldots ,r\}$ occurs in $w$; $(n_1,\ldots ,n_r)$ is its Parikh vector, the $w_{i,1}\cdots w_{i,n}$ are its conjugates. 

The first lemma states that $B$ is injective on the conjugacy classes, which is proved for example in \cite{etc} or \cite{rr1}; we give here a short proof for sake of completeness.

\begin{lm}\label{inj} Every antecedent of $B(w)$ by the Burrows-Wheeler transform is conjugate to $w$.\end{lm}
\begin{proof}
In the array of the conjugates of $w$, each column word $w_{1,j}\cdots w_{n,j}$ has the same Parikh vector as $w$, so we retrieve this vector from $B(w)$; thus we know the first column word, which is $1^{n_1}\ldots r^{n_r}$, and the last column word which is $B(w)$. Then the words $w_{n,j}w_{1,j}$ are precisely all  words of length $2$ occurring in the conjugates of $w$, and by ordering them we get the first two columns of the array. Then  $w_{n,j}w_{1,j}w_{2,j}$ constitute all  words of length $3$ occurring in the conjugates of $w$, and  we get also the subsequent column, and so on until we have retrieved the whole array, thus $w$ up to conjugacy.\end{proof}

It is easy to see  that $B,$ viewed as a mapping from words to words, is not surjective (see for instance \cite{rr1}). A more precise result will be proved in Corollary \ref{srj} below.

\begin{lm}\label{int} If $w$ is $\pi$-clustering, the mapping $w_{1,j} \mapsto w_{n,j}$ defines a discrete $r$-interval exchange transformation
with length vector $(n _1,n _2,\ldots ,n_r)$,
 and
permutation $\pi$. \end{lm}
\begin{proof} We order the occurrences of each letter in $w$ by putting $w_i<w_j$ if the conjugate\\ $w_i\cdots w_nw_1\cdots w_{i-1}$ is lexicographically smaller than $w_j\cdots w_nw_1\cdots w_{j-1}.$ By primitivity, the $n$ letters of $w$ are uniquely ordered as $$1_1< \cdots <1_{n_1}<2_1<\cdots <2_{n_2}<\cdots <r_1<\cdots <r_{n_r},$$ and the first column word is $1_1\cdots 1_{n_1}2_1 \cdots 2_{n_2}\cdots r_1\cdots r_{n_r}$. We look at the last column word: if $w_{n,j}$ and $w_{n,j+1}$ are both some letter $k,$ the order between these two occurrences of $k$ is given by the next letter in the conjugates of $w$, and these are respectively $w_{1,j}$ and $w_{1,j+1}.$ Thus $w_{n,j}<w_{n,j+1}.$  Together with the hypothesis, this implies that the last column word is $$(\pi 1)_1\cdots (\pi 1)_{n_{\pi 1}}\cdots  (\pi r)_1\cdots (\pi r)_{n_{\pi r}}.$$ Thus, if we regard the rule  $w_{1,j} \mapsto w_{n,j}$ as a mapping on the $n_1+\ldots+n_r$ points $$\{1_1,\ldots ,1_{n_1},2_1, \ldots 2_{n_2}, \ldots ,r_1,\ldots ,r_{n_r}\},$$  and put $\Delta_i=\{i_1,\ldots ,i_{n_i}\}$, we get the claimed result. \end{proof}

\begin{cl}\label{srj} If the discrete $r$-interval exchange transformation
 $T$
with length vector $(n _1,n _2,\ldots ,n_r)$,
 and
permutation $\pi$  is not minimal, the word $(\pi 1)^{n_{\pi 1}}\ldots (\pi r)^{n_{\pi r}}$ has no primitive antecedent by the Burrows-Wheeler transform. \end{cl}
\begin{proof} Let $w$ be such an antecedent. By the previous lemma, the map $w_{1,j} \mapsto w_{n,j}$ corresponds to $T$.
If $T$ is not minimal, there is a proper subset $E$ of $\{1_1,\ldots ,1_{n_1},2_1, \ldots 2_{n_2}, \ldots ,r_1,\ldots ,r_{n_r}\}$ which is invariant by  $w_{1,j} \mapsto w_{n,j}.$  Thus in the conjugates of $w$, preceding any occurrence of a letter of $E$ is another occurrence of a letter of $E.$ This implies that $w$ is made up entirely  of letters of $E,$ a contradiction. \end{proof}

 \begin{proof}{($(1)$ implies  $(2)$)} Let $w$ be as in the hypothesis. Then $B(w)=(\pi 1)^{n_{\pi 1}}\cdots (\pi r)^{n_{\pi r}}$. Thus  the transformation $T$ of Lemma \ref{srj} is minimal, and thus has a periodic trajectory $w'w'w'\ldots$, where $w'$ has Parikh vector $(n_1,\ldots ,n_r)$. If $w'=u^{k}$, then $n_i=kn'_i$ for all $i,$  and the set made with the $n'_i$ leftmost points of each $\Delta_i$ is $T$-invariant, thus $w'$ must be primitive.  
 
By the proof, made above,  that  $(2)$ implies $(1)$, $w'$ is $\pi$-clustering. Hence $B(w')=B(w)$ and, by Lemma \ref{inj}, $w$ is conjugate to $w'$, hence $ww$ occurs also in a trajectory of $T$. \end{proof}

Some of the hypotheses of Theorem \ref{t1} may be weakened.\\

 {\bf Alphabet}. $\{ 1,\ldots ,r\}$ can be replaced by by any ordered set $A=\{a_1<a_2< \cdots <a_r\}$ by using a letter-to-letter morphism. Thus for a given word $w$, we can restrict the alphabet to the letters occurring in $w$. Note that if $ww$ occurs in a trajectory of an $r$-interval exchange transformation, but only the letters $j_1,\ldots ,j_d$ occur in $w$, then, by  the reasoning of the proof that $(4)$ implies $(1)$, $w$ is $\pi'$-clustering, where $\pi'$ is the unique permutation on $\{1,\ldots , d\}$ such that  $(\pi')^{-1}(y)< (\pi')'^{-1}(z)$ iff $\pi^{-1}(j_y)< \pi^{-1}(j_z)$. If $\pi$ is a permutation defining perfect clustering, then so is $\pi'$.\\

{\bf Primitivity}. The Burrows-Wheeler transformation can be extended to a non-primitive word $w_1\cdots w_n$, by ordering its $n$ (non necessarily different) conjugates $w_i\cdots w_nw_1\cdots w_{i-1}$ by non-strictly increasing lexicographical order and taking the word made by their last letters.

In this case the result of Lemma \ref{srj} does not extend: For example $B(1322313223)=3333222211$ though the discrete $3$-interval exchange transformation
with length vector $(2,2,4)$,
 and
permutation $\pi 1=3, \pi 2=2, \pi 3=1$  is not minimal. Note that if $(\pi 1)^{n_{\pi 1}}\cdots (\pi r)^{n_{\pi r}}$ has a non-primitive antecedent by the Burrows-Wheeler transform, then the $n_i$ have a common factor $k.$ There exist (see below) non-minimal discrete interval exchange transformations which do not satisfy that condition, and thus words such as $32221$ which have no antecedent at all by the Burrows-Wheeler transformation.

But our Theorem \ref{t1} is still valid for non-primitive words: the proof in the first direction does not use the primitivity, while in the reverse direction we  write $w=u^k$, apply our proof to the primitive $u$, and check that $u^{2k}$ occurs also in a trajectory.\\

{\bf Two permutations}. An extension of Theorem \ref{t1} which fails is to consider, as the dynamicians do \cite{v}, interval exchange transformations defined by permutations $\pi$ and $\pi'$; this amounts to coding the interval $\Delta_i$ by $\pi' i$ instead of $i$.  A simple counter-example will be clearer than a long definition: take points $x_1,\ldots ,x_9$ labelled $223331111$ and send them to $111133322$ by a (minimal) discrete $3$-interval exchange transformation, but where the points are not labelled  as in Definition 3 (namely $Tx_1=x_8$, $Tx_3=x_5$ etc...). Then $w=123131312$ is such that $ww$ occurs in trajectories of $T$ but $B(w)=323311112$.

\section{Building clustering words}\label{s3}

Theorem \ref{t1} provides two different ways to build clustering words, from infinite trajectories either of discrete (or rational) interval exchange transformations or of continuous aperiodic interval exchange transformations. For $r=2$ and the permutation  $\pi 1=2, \pi 2=1$, the first ones give all the periodic balanced words, and the second ones gives  (by Proposition \ref{idoc} below) all infinite Sturmian words:  both these ways of building clustering words on two letters are used, explicitly or implicitly, in   \cite{jz}.\\

The use of discrete interval exchange transformations leads naturally to the question of characterizing all minimal discrete $r$-interval exchange transformations through their length vector; this has been solved by \cite{pr} for $n=3$ and $\pi 1=3, \pi 2=2, \pi 3=1$: if the length vector  is $(n_1,n_2,n_3)$, minimality is equivalent to $(n_1+n_2)$ and $(n_2+n_3)$ being coprime. Thus 
 \begin{ex} The discrete interval exchange  $111122333 \to 333221111$, gives rise to the perfectly clustering word $122131313.$ \end{ex}
The same reasoning  extends to other permutations: for $\pi 1=2, \pi 2=3, \pi 3=1$, minimality is equivalent to $n_1$ and $(n_2+n_3)$
 being coprime; for $\pi 1=3, \pi 2=1, \pi 3=2$, minimality is equivalent to $n_3$ and $(n_2+n_1)$
 being coprime; for other permutation on these three letters, $T$ is never minimal.  
 
 For $r\geq 4$ intervals, the question is still open. An immediate equivalent condition for non-minimality is $\sum_{i=1}^m s_{w_i}=0$ for $m<n_1+\cdots + n_r$ and $w_1\cdots w_m$ a word occurring in a  trajectory. It is easy to build non-minimal examples satisfying such an equality for simple words $w$, for example for $r=4$ and $\pi 1=4, \pi 2=3, \pi 3=2, \pi 4=1$,  $n_1=n_2=n_3=1$ gives non-minimal examples for any value of $n_4$, the equality being satisfied for $w=24^q$ if $n_4=3q$, $w=14^{q+1}$ if $n_4=3q+1$, $w=34^q$ if $n_4=3q+2$. Similarly, the following example shows how we still do get clustering words, but they may be  somewhat trivial. 
 \begin{ex} The discrete interval exchange  $111233444 \to 444332111$ satisfies the above equality for $w=14$; it is non-minimal and gives two perfectly clustering words on smaller alphabets, $41$ and $323$. \end{ex}

To study continuous aperiodic interval exchange transformations we need a technical condition called {\em i.d.o.c.} \cite{kea} which states that {\em the orbits of the discontinuities of $T$ are infinite and disjoint}. It is proved in \cite{kea} or in \cite{v} that this condition implies aperiodicity and minimality, and that, if $\pi$ is {\em primitive}, i.e. $\pi \{1,\ldots ,d\} \neq \{1,\ldots ,d\}$ for $d<r$, then the $r$-interval exchange transformation
with probability vector 
$(\alpha _1,\ldots ,\alpha _r)$
 and
permutation $\pi$ satisfies the i.d.o.c. condition if $\alpha_1$, \ldots , $\alpha_r$ and $1$ are rationally independent. We can now prove

\begin{pp}\label{idoc} Let $w=w_1\cdots w_n$ be a primitive word on $A=\{ 1,\ldots ,r\}$, such that every letter of $A$ occurs in $w$. Then  $w$ is $\pi$-clustering if and only if 
 $ww$ occurs in a trajectory of a continuous $r$-interval exchange transformation with permutation $\pi$, satisfying the i.d.o.c. condition.
\end{pp} 
\begin{proof} The ``if" direction is as in Theorem \ref{t1}. To get the ``only if", we generate $w$ by a minimal discrete
interval exchange transformation as in $(2)$ of Theorem  \ref{t1}, and thus $\pi$ is primitive. Then we replace it by a continuous periodic interval exchange transformation as in the proof that $(3)$ implies $(4)$. But, because cylinders are always semi-open intervals, if a given word $ww$ occurs in a trajectory of a continuous  $r$-interval exchange transformation
 with permutation $\pi$ and  probability vector 
$(\alpha _1,\ldots ,\alpha _r)$, it occurs also in trajectories of every $r$-interval exchange transformation
 with the same permutation whose  probability vector is close enough to
$(\alpha _1,\ldots ,\alpha _r).$ Thus we can change the $\alpha_i$ to get the irrationality condition which implies the i.d.o.c. condition. \end{proof}

Trajectories of interval exchange transformations satisfying the i.d.o.c. condition
may be explicitly constructed via the {\em self-dual induction} algorithms of \cite{fhz2} for $r=3$ and $\pi 1=3, \pi 2=2, \pi 3=1$, \cite{fz1} for all $r$ and $\pi i=r+1-i$, and the forthcoming \cite{fie} in the most general case. More precisely, Proposition 4.1 of \cite{fz1} shows that if the permutation is $\pi i=r+1-i$ (or more generally if the permutation is in the  {\em Rauzy class} of $\pi i=r+1-i$), then there exist infinitely many words $ww$ in the trajectories. It also gives a sufficient condition for building such words: if a {\em bispecial} word $w$, a suffix $s$ and a prefix $p$ of $w$ are such that $pw=ws$, then both $pp$ and $ss$ occur in the trajectories. In turn, a recipe to achieve that relation is given in $(i)$ of Theorem 2.8 of \cite{fz1}: we just need that in the underlying algorithm described in Section 2.6 of \cite{fz1}, either $p_n(i)=i$ or $m_n(i)=i$ (except for some initial values of $n$, where, for $i=1$, $p$ and $s$ are longer than $w$). Many explicit examples of $ww$ have been built in this way.
  \begin{itemize}
\item For $r=3$  in \cite{fhz2}, $w=A_k$, $w=B_k$ in Proposition 2.10,
\begin{ex}  $13131312222$ and $122131222131221313$ are perfectly clustering.\end{ex}
\item For $r=4$ in \cite{fz2}, $w=M_2(k)$, $w=P_3(k)M_1(k)$ in Lemma 4.1 and in Lemma 5.1,
\begin{ex} $2^m(3141)^n32$ are perfectly clustering for any $m\geq 3$ and $n\geq 2$. \end{ex}
\item For all $r=n$ in \cite{ngon}, $w=P_{k,1,1}$, $w=P_{k,n-i,i+1}P_{k,i+1,n-i}$, $w=M_{k,n+1-i,i-1}M_{k,i-1,n+1-i}$ in Theorem 12; \begin{ex} $5252434252516152516161525161$ is perfectly clustering. \end{ex}.
  \end{itemize} 
  
 For other permutations, we shall describe in \cite{fie}  an algorithm generalizing the one in \cite{fz1}.  We also construct an example of an interval exchange transformation which does not produce infinitely many $ww$. For the permutation $\pi 1=4, \pi 2=3, \pi 3=1, \pi 4=2$, examples can be found in Theorem 5.2 of \cite{fz1}, with $w=P_{1,q_n}M_{2,q_n}$, $w=P_{2,q_n}M_{3,q_n}$,  $w=P_{3,q_n}M_{1,q_n}$, \begin{ex} $4123231312412$ is $\pi$-clustering,\end{ex} We remark that our self-dual induction algorithms for aperiodic  interval exchange transformations generate families of nested clustering words with increasing length, and thus may be more efficient in producing very long clustering words than the more immediate algorithm using discrete interval exchange transformations.

\end{document}